\pdfoutput=1
\documentclass{svmult}

\usepackage{type1cm}          
\usepackage{graphicx}         
\usepackage{mathtools}        
\usepackage{booktabs}         
\usepackage{newtxtext}        
\usepackage[varvw]{newtxmath} 

\usepackage[hidelinks]{hyperref}
\usepackage{cleveref}

\providecommand{\Description}[1]{}

\crefname{proposition}{Proposition}{Propositions}
\crefname{corollary}{Corollary}{Corollaries}
\crefname{remark}{Remark}{Remarks}
\crefname{definition}{Definition}{Definitions}
\Crefname{proposition}{Proposition}{Propositions}
\Crefname{corollary}{Corollary}{Corollaries}
\Crefname{remark}{Remark}{Remarks}

\begin{document}

\title*{Heilbronn's Problem in the Unit Triangle: Certified
Optimal Configurations for \texorpdfstring{$n \le 8$}{n <= 8}}
\titlerunning{Heilbronn's Problem in the Unit Triangle}
\author{Nathan Sudermann-Merx}
\authorrunning{N.\ Sudermann-Merx}
\institute{Nathan Sudermann-Merx \at Institute for Computer Science,
DHBW Mannheim, Coblitzallee 1--9, 68163 Mannheim, Germany,
\email{nathan.sudermann-merx@dhbw.de}}

\maketitle

\abstract*{We study Heilbronn's triangle problem in the unit right triangle, where $n$
points are placed to maximize the smallest of the $\binom{n}{3}$ triangle areas
they span. We prove a boundary-structure result: unless all three vertices are
occupied, some optimal configuration with $n \ge 5$ has at least four points on
the boundary, one edge carrying two of them. With the affine $S_3$ symmetry this
fixes four boundary points and $n$
orientation variables in a mixed-integer model that certifies global optimality
for all $n \le 8$: for $n = 8$ apparently the first proof, and for $n = 7$ an
independent confirmation of the symbolic-computation proof of Zeng and Chen.
For $n \le 7$ we obtain exact optima with explicit configurations. For $n = 8$
the optimum is conjectured to be the real root of a septic obtained by Chen,
Zeng and Zhou, which our reconstruction confirms to $250$ digits. We show its
Galois group is $S_7$, so on that conjecture no expression in radicals exists.}

\abstract{We study Heilbronn's triangle problem in the unit right triangle, where $n$
points are placed to maximize the smallest of the $\binom{n}{3}$ triangle areas
they span. We prove a boundary-structure result: unless all three vertices are
occupied, some optimal configuration with $n \ge 5$ has at least four points on
the boundary, one edge carrying two of them. With the affine $S_3$ symmetry this
fixes four boundary points and $n$
orientation variables in a mixed-integer model that certifies global optimality
for all $n \le 8$: for $n = 8$ apparently the first proof, and for $n = 7$ an
independent confirmation of the symbolic-computation proof of Zeng and Chen.
For $n \le 7$ we obtain exact optima with explicit configurations. For $n = 8$
the optimum is conjectured to be the real root of a septic obtained by Chen,
Zeng and Zhou, which our reconstruction confirms to $250$ digits. We show its
Galois group is $S_7$, so on that conjecture no expression in radicals exists.}

\section{Introduction}
\label{sec:intro}
Heilbronn's triangle problem asks for the placement of $n$ points inside a
convex region of unit area so as to maximize the minimum area among all
triangles formed by triples of these points. Areas transform predictably under
affine maps of the domain, but a triangle is not an affine image of a square,
so exact results for one domain do not carry over to the other.

Most exact results concern the unit square. In our companion
paper~\cite{SudermannMerx2026square} we developed a mixed-integer quadratically
constrained programming (MIQCP) framework certifying global optimality there up
to $n = 9$, together with exact coordinate reconstructions, and we refer to it
for background and a literature review.

Here we apply that framework to the \emph{unit right triangle}
$\mathcal{T} = \{(x,y) \in \mathbb{R}^2 : x \ge 0,\ y \ge 0,\ x + y \le 1\}$,
the unit simplex in $\mathbb{R}^2$, of area $\tfrac12$, and write
$\Delta_n^{\triangle}$ for the largest achievable minimum triangle area for $n$
points in~$\mathcal{T}$. Our structural result (\Cref{prop:four-on-boundary}) is
that for $n \ge 5$, unless all three vertices are occupied, \emph{some} optimal
configuration places at least \emph{four} points on $\partial\mathcal{T}$, and
that the affine $S_3$ symmetry (\Cref{cor:flush-edge}) allows the edge carrying
two of them to be fixed in advance. The model built on this certifies global
optimality for all $n \le 8$ and recovers exact optima for $n \le 7$: for
$n = 7$ an independent confirmation of the symbolic proof of Zeng and
Chen~\cite{ZengChen2019}, for $n = 8$, listed there as unsolved, to our
knowledge the first optimality proof. \Cref{sec:arithmetic} then confirms the
septic conjectured for $n=8$ by Chen, Zeng and Zhou~\cite{ChenZengZhou2014} and
shows that its Galois group is $S_7$, so that on their conjecture no expression
in radicals exists.

\section{Preliminaries}
\label{sec:prelim}

\subsection{Relation to the equilateral triangle}
\label{sec:prelim:equilateral}
The best-known configurations are recorded at Erich's Packing
Center~\cite{FriedmanPackingCenter} using an equilateral triangle of
\emph{unit area}, whereas our model uses~$\mathcal{T}$, of area~$\tfrac12$.
Triangle areas scale by the Jacobian modulus of the affine bijection between
the two domains, so passing from~$\mathcal{T}$ to their domain \emph{doubles}
the optimal value.

\subsection{Boundary structure}
\label{sec:prelim:boundary}
The boundary $\partial\mathcal{T}$ has three edges: the \emph{legs}
$e_1 = \{y = 0\}$ and $e_2 = \{x = 0\}$, and the \emph{hypotenuse}
$e_3 = \{x + y = 1\}$.

\begin{proposition}[Four points on the boundary]
\label{prop:four-on-boundary}
  Let $n \ge 5$ and let $P^\ast$ be an optimal configuration whose hull
  $K := \operatorname{conv}(P^\ast)$ differs from~$\mathcal{T}$. Then some
  area-preserving affine image of $P^\ast$ lies in~$\mathcal{T}$, is again
  optimal, and has at least four points on $\partial\mathcal{T}$, one edge
  of~$\mathcal{T}$ carrying two of them and each of the other two edges carrying
  at least one.
\end{proposition}

\noindent
We use two facts about minimum-area triangles enclosing a convex polygon~$Q$.
\textbf{(M)} In \emph{every} minimum-area triangle enclosing~$Q$, the
\emph{midpoint} of each side lies on~$Q$~\cite{KleeLaskowski1985,ORourke1986}.
\textbf{(F)} The minimum area is attained by a triangle with one side
\emph{flush} with an edge of~$Q$, i.e.\ containing that edge, which is what
makes the edge-enumerating algorithm of~\cite{ORourke1986} correct.

Fact~(F) is an \emph{existence} statement --- $\mathcal{T}$ itself need not be
flush with~$K$ --- and this alone forces the passage to an affine image. The
distinction does not affect the model, which need only leave one optimum
intact.

\begin{proof}
  \emph{Step 1: $\mathcal{T}$ is a minimum-area triangle containing~$K$.} If some
  triangle $\mathcal{T}' \supseteq K$ had area below~$\tfrac12$, the affine
  bijection $\psi\colon \mathcal{T}' \to \mathcal{T}$ would have
  $|\det\psi| = \tfrac12 / \operatorname{area}(\mathcal{T}') > 1$, so
  $\psi(P^\ast) \subseteq \mathcal{T}$ would satisfy
  $\Delta(\psi(P^\ast)) = |\det\psi|\,\Delta(P^\ast) > \Delta(P^\ast)$,
  contradicting optimality.

  \emph{Step 2: transport to a flush triangle.} By~(F) pick a minimum-area
  triangle $\mathcal{T}''\supseteq K$ with a side flush with an edge of~$K$.
  Since $\operatorname{area}(\mathcal{T}'') = \operatorname{area}(\mathcal{T})
  = \tfrac12$ by Step~1, the affine bijection
  $\varphi\colon \mathcal{T}'' \to \mathcal{T}$ has $|\det\varphi| = 1$ and thus
  preserves every area. Hence $P' := \varphi(P^\ast) \subseteq \mathcal{T}$ is
  again optimal, $\mathcal{T}$ is a minimum-area triangle containing
  $K' := \operatorname{conv}(P') = \varphi(K)$, and one side of~$\mathcal{T}$ is
  flush with an edge of~$K'$. Also $K' \neq \mathcal{T}$: a triangular $K$ would
  be its own minimum-area enclosing triangle, forcing
  $\operatorname{area}(K) = \tfrac12$ and $K = \mathcal{T}$, which is excluded,
  so $K$ is not a triangle and $K \subsetneq \mathcal{T}''$.

  \emph{Step 3: contacts.} By~(M) applied to $Q = K'$, the midpoint of every
  side of~$\mathcal{T}$ lies on~$K'$. Each side spans a supporting line
  of~$K'$, so it meets $K'$ in a face of~$K'$, and that face contains the
  midpoint: it is either a single vertex, which is then the midpoint itself, or
  an entire edge.

  \emph{Step 4: counting.} Let $f \ge 1$ be the number of flush sides. A
  non-flush side contributes the vertex at its midpoint, which lies in the
  relative interior of that side and hence on no other side. A flush side
  contributes the two endpoints of its edge, and two flush sides share at most
  the corner of~$\mathcal{T}$ between them. Were $f = 3$ with all three corners
  vertices of~$K'$, then $K' \supseteq \mathcal{T}$, excluded, so at most two
  corners are vertices of~$K'$. Thus $f = 1$ yields $2+1+1$, $f = 2$ at least
  $3+1$, and $f = 3$ at least $6-2$ distinct vertices of~$K'$ on
  $\partial\mathcal{T}$, in every case at least four, one side carrying two. The
  vertices of $K' = \operatorname{conv}(P')$ are points of~$P'$.
\end{proof}

\begin{remark}[The remaining case]\label{rem:case-split}
  As $\mathcal{T}$ is the convex hull of its vertices, $K = \mathcal{T}$ holds
  exactly when $P^\ast$ contains all three of them, so the hypothesis of
  \Cref{prop:four-on-boundary} leaves exactly one case open, and we settle it
  separately. A model fixing the three vertices, with the other $n-3$ points
  free in~$\mathcal{T}$, certifies the optima $0.06699$, $0.05647$, $0.03479$
  and $0.02740$ for $n = 5, \ldots, 8$, all strictly below the values of
  \Cref{tab:summary}, so no optimum with $n \ge 5$ has $K = \mathcal{T}$ and
  the fixing below loses nothing. The optima carry four ($n=5$), five
  ($n=6,8$) or six ($n=7$) boundary points.
\end{remark}

\begin{corollary}[The doubly occupied edge may be fixed]
\label{cor:flush-edge}
  The six affine maps permuting the vertices of~$\mathcal{T}$ send $\mathcal{T}$
  onto itself with Jacobian modulus~$1$. Being area-preserving, and the
  objective affine-covariant, they realize $S_3$ acting on $\{e_1,e_2,e_3\}$.
  So the edge carrying two boundary points in~\Cref{prop:four-on-boundary}
  may be taken to be~$e_1$, and an optimal configuration is then labelled with
  two points on~$e_1$ ($y = 0$), one on~$e_2$ ($x = 0$), and one on~$e_3$
  ($x + y = 1$).
\end{corollary}

\section{The Optimization Model}
\label{sec:model}
Write $p_i = (x_i,y_i)$ for $i \in I := \{0,\dots,n-1\}$, let $T$ be the set of
triples, and let $A_{ijk} = \tfrac12(x_i(y_j-y_k) + x_j(y_k-y_i) + x_k(y_i-y_j))$
be the \emph{signed} area of $p_ip_jp_k$, whose area is $|A_{ijk}|$. Maximizing
the smallest of these over~$\mathcal{T}$ reads
\begin{equation}\label{eq:model}
  P_\triangle^\star:\quad \max\ z \quad\text{s.t.}\quad
  z \le (2b_t-1)\,A_t\ \ (t \in T), \qquad
  x_i,\, y_i \ge 0,\quad x_i + y_i \le 1\ \ (i \in I),
\end{equation}
with $z \ge 0$ and $b_t \in \{0,1\}$. The binary $b_t$ selects the orientation,
giving $z \le A_t$ for $b_t = 1$ and $z \le -A_t$ otherwise, so $z \le |A_t|$
throughout. As $A_t$ is bilinear, \eqref{eq:model} is a mixed-integer
quadratically constrained program, solved to proven global optimality by
Gurobi~13, the products being handled by the substitution $w_{ij} = x_i y_j$.

The main gain comes from the boundary structure. For $n \ge 5$,
\Cref{cor:flush-edge} lets us put
$p_0, p_1$ on $e_1$, $p_2$ on $e_3$ and $p_3$ on $e_2$, labelled
counterclockwise, removing four continuous variables, and it breaks the residual
reflection $(x,y)\mapsto(1-x-y,y)$ by $x_0 + x_1 \le 1$. In that labelling the
four boundary points are in convex position, so every triangle among them, and
every one formed by $p_0, p_1$ and a free point above $e_1$, is positively
oriented. These $\binom{4}{3} + (n-4) = n$ triples make up the set
$T^+$, as in the square model~\cite{SudermannMerx2026square}.

\section{Certified Optimal Configurations}
\label{sec:results}
\Cref{tab:summary} reports the certified optima, runtimes and exact
coordinates (numerical for $n=8$), and \Cref{fig:configs} shows the
configurations. All runs used Gurobi~13 with a MIP gap tolerance of $10^{-4}$.

\begin{table}[t]
  \caption{Certified optima, runtimes, and exact representative configurations.
  For $n=5,6$ the optimum is a family and one representative is shown. For
  $n=8$ the value and the coordinates are the certified numerical solution.}
  \Description{A table with one row per $n$ from 3 to 8 and four columns: $n$,
  the certified optimum $\Delta_n^\triangle$, the solver runtime in seconds, and
  the coordinates of a representative optimal configuration. The optima are
  $\tfrac12$, $\tfrac16$, $\tfrac32-\sqrt2$, $\tfrac1{16}$, $\tfrac{7}{144}$ and
  approximately $0.033896$, with runtimes below one second, below one second,
  $0.2$, $3.6$, $22.6$ and $2324$ seconds.}
  \label{tab:summary}
  \centering\footnotesize
  \setlength{\tabcolsep}{4pt}
  \begin{tabular}{c l r l}
  \toprule
  $n$ & $\Delta_n^{\triangle}$ & Time (s) & optimal points $(x_i, y_i)$ \\
  \midrule
  $3$ & $\tfrac12$ & $<1$ & $(0,0),\ (1,0),\ (0,1)$ \\
  $4$ & $\tfrac16$ & $<1$ & $(0,0),\ (1,0),\ (0,1),\ (\tfrac13,\tfrac13)$ \\
  $5$ & $\tfrac32-\sqrt2$ & $0.2$ & $(0,0),\ (2{-}\sqrt2,0),\ (0,1),\ (2{-}\sqrt2,\sqrt2{-}1),\ (3{-}2\sqrt2,\sqrt2{-}1)$ \\
  $6$ & $\tfrac1{16}$ & $3.6$ & $(\tfrac14,0),\ (\tfrac34,0),\ (0,\tfrac38),\ (\tfrac34,\tfrac14),\ (0,1),\ (\tfrac14,\tfrac12)$ \\
  $7$ & $\tfrac{7}{144}$ & $22.6$ & $(\tfrac16,0),\ (\tfrac34,0),\ (0,\tfrac14),\ (\tfrac56,\tfrac16),\ (0,\tfrac56),\ (\tfrac14,\tfrac34),\ (\tfrac13,\tfrac13)$ \\
  $8$ & $\approx 0.033896$ & $2324$ & $(0,0),\ (0.889,0),\ (0.721,0.279),\ (0,0.845),$ \\
      &                    &        & $\quad(0.112,0.283),\ (0.112,0.888),\ (0.395,0.393),\ (0.439,0.076)$ \\
  \bottomrule
  \end{tabular}
\end{table}

\begin{figure}[t]
  \centering
  \includegraphics[width=0.50\linewidth]{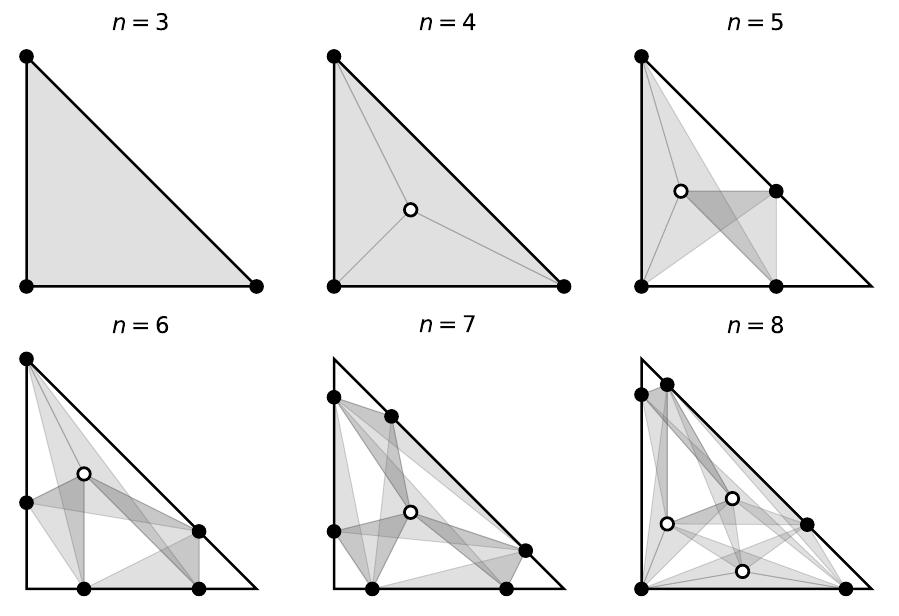}
  \caption{Certified optima for $3 \le n \le 8$. Filled dots lie on
  $\partial\mathcal{T}$, open dots inside, shaded triangles attain the minimum
  area. For $n=5,6$ one representative is shown.}
  \Description{Six panels, one per $n$ from 3 to 8, each drawing the unit right
  triangle with the $n$ points of a certified optimal configuration. Points on
  the boundary are filled dots, points in the interior are open dots, and the
  triangles attaining the minimum area are shaded gray, overlapping shadings
  appearing darker. The counts of boundary points are three for $n=3$ and
  $n=4$, four for $n=5$, five for $n=6$, six for $n=7$ and five for $n=8$, so
  every configuration with $n\ge5$ has at least four, as
  \Cref{prop:four-on-boundary} predicts. The number of interior points is zero
  for $n=3$, one for $n=4,5,6,7$ and three for $n=8$.}
  \label{fig:configs}
\end{figure}

For $n \le 7$ the optimum is attained in \emph{closed form}: the configuration
listed in \Cref{tab:summary} is feasible and meets the stated value exactly.
The cases $n=3,4$ are classical. For $n=5$ the optimum $\tfrac32-\sqrt2$ is
\emph{not unique}: a one-parameter family arises
by sliding the second $e_1$-point $(a,0)$ over $a\in[\,2-\sqrt2,\ \sqrt2/2\,]$,
and we list the endpoint $a=2-\sqrt2$. For $n=6$ the optimum $\tfrac1{16}$ is
likewise a two-parameter family (a rational representative is shown), while for
$n=7$ it is $\tfrac{7}{144}$, on the $\tfrac1{12}$-grid. For $n=8$ no closed
form is known, and \Cref{sec:arithmetic} explains what stands in the way.

Our values agree with the configurations recorded at Erich's Packing
Center~\cite{FriedmanPackingCenter} (equilateral triangle, unit area): doubled,
they give $3-2\sqrt2$, $\tfrac18$, $\tfrac{7}{72}$ and $0.06779\ldots$ for
$n=5,\dots,8$. Peng proved optimality for $n=5$ in
1989~\cite[Sect.~9.3]{Soifer2009}, Yang, Zhang and Zeng for
$n=5,6$~\cite{YangZhangZeng1994}. The configurations for $n=7,8$ are
Cantrell's. For $n=7$, after the grid-search upper bound $23/200$
of~\cite{DeComiteDelahaye2009}, Zeng and Chen~\cite{ZengChen2019} proved
optimality of $7/72$, with uniqueness, by an eight-case symbolic analysis built
on that localization; our certificate reaches the same configuration
independently of both. For $n=8$ only bounds were known: a branch-and-bound
over octuples of subtriangles reached $0.067816$, within $2.7\cdot10^{-5}$ of
Cantrell's $0.06778921\ldots$, at some $11{,}000$ CPU hours on an $80$-core
cluster~\cite{ChenZengZhou2014}, and the case remained open
in~\cite{ZengChen2019}. Model~\eqref{eq:model} settles it, and re-proves
$n=7$, in $2324$ and $22.6$ seconds on one machine.

\subsection{The case \texorpdfstring{$n=8$}{n=8}}
\label{sec:arithmetic}
For $n \le 7$ the optimal coordinates are quadratic at worst, rational except
when $n = 5$. For $n = 8$ the ground has been prepared by Chen, Zeng and
Zhou~\cite{ChenZengZhou2014}: they proved that at most eleven triangles of an
optimal configuration can be minimal, namely a specific set $T_{11}$ of index
triples, and, assuming all eleven to be minimal, eliminated the ten free
coordinates from the resulting equal-area system \emph{exactly} by a Gr\"obner
basis. This gives a septic whose unique real root they conjecture to be the
optimum. Rescaled to our normalization,
\begin{align*}
  P(t) ={}& 1216512\,t^7 + 1022976\,t^6 - 2745456\,t^5 + 2218730\,t^4 \\
          & {}+ 32206304\,t^3 - 19148615\,t^2 + 3156103\,t - 86233 .
\end{align*}
Our certificate supports the conjecture from the other side: solving the same
system to several hundred digits, seeded at the certified optimum, reproduces
the real root of~$P$ to over $250$ digits, and our critical triangles are
exactly~$T_{11}$. Their localization also feeds back into the model, confining
each point to a box of side $\tfrac1{32}$. That makes all $\binom83 = 56$
orientations constant over the feasible set, so every binary can be fixed in
advance and $P_\triangle^\star$ becomes a continuous program, certifying $n=8$
in $0.04$ seconds rather than $2324$ at the price of inheriting their
Theorem~2, which the run of \Cref{tab:summary} does not.

The arithmetic consequence does not seem to have been noted. Now $P$ is
primitive and irreducible over~$\mathbb{Q}$, already modulo~$23$, so its Galois
group $G$ is a transitive subgroup of~$S_7$. By Dedekind's
theorem~\cite{Cohen1993} the degrees of the irreducible factors of $P$ modulo an
unramified prime form the cycle type of a Frobenius element of~$G$. Modulo~$7$
these are $(1,1,5)$, giving an element of order~$5$, so $5 \mid |G|$, and
modulo~$5$ they are $(1,6)$, an odd permutation, so
$G \not\subseteq A_7$. Of the transitive subgroups of $S_7$, whose orders are
$7,14,21,42,168,2520,5040$, only $|A_7|$ and $|S_7|$ are divisible by~$5$, hence
$G = S_7$. As $S_7$ is not solvable, no root of $P$ lies in a radical extension
of~$\mathbb{Q}$~\cite{Stewart2015}. So if the conjecture
of~\cite{ChenZengZhou2014} holds, then $\Delta_8^\triangle$ and the optimal
coordinates admit \emph{no} expression in radicals --- in contrast to the unit
square, whose optima for $n=8,9$ are
quadratic~\cite{SudermannMerx2026square}.

\section{Conclusion and Outlook}
\label{sec:conclusion}
We have certified global optimality for Heilbronn's triangle problem in the
unit right triangle for all $n \le 8$, apparently the first proof for $n = 8$
and an independent one for $n = 7$, and determined the optimum exactly for
$n \le 7$. The key
ingredient is the boundary structure of \Cref{prop:four-on-boundary}:
some optimum carries four points on $\partial\mathcal{T}$, and the affine $S_3$
symmetry lets us fix those four points and $n$ orientation variables in a
single self-contained MIQCP.

The optima are varied. Those for $n = 5, 6$ form continuous families, the
optimum $\tfrac{7}{144}$ for $n = 7$ lies on a rational grid, and
$\Delta_8^\triangle$ is conjecturally a degree-$7$ algebraic number whose Galois
group we show to be $S_7$, ruling out any expression in radicals
(\Cref{sec:arithmetic}). What remains open for $n=8$ is narrow but real: one
must show that all eleven triangles of $T_{11}$ are simultaneously minimal, of
which~\cite{ChenZengZhou2014} proves the upper half. Beyond that, certificates
for $n \ge 9$ and a characterization of the optimal families are the natural
next steps.


\end{document}